\documentclass[12pt,draft]{amsart}

\usepackage[margin=1in]{geometry}
\usepackage{amsmath,amssymb,amsthm,url}

\title{Bounded gaps between primes in number fields and function fields}

\author[A. Castillo]{Abel Castillo}
\address{Department of Mathematics \\ University of Illinois at Chicago \\ Chicago, IL 60607}
\email{acasti8@uic.edu}

\author[C. Hall]{Chris Hall}
\address{Department of Mathematics \\ University of Wyoming \\ Laramie, WY 82071}
\email{chall14@uwyo.edu}

\author[R. Lemke Oliver]{Robert J. Lemke Oliver}
\address{Department of Mathematics \\ Stanford University \\ Palo Alto, CA 94305}
\email{rjlo@stanford.edu}

\author[P. Pollack]{Paul Pollack}
\address{Department of Mathematics\\University of Georgia\\Athens, GA 30602}
\email{pollack@uga.edu}

\author[L. Thompson]{Lola Thompson}
\address{Department of Mathematics\\Oberlin College\\Oberlin, OH 44074}
\email{lola.thompson@oberlin.edu}

\thanks{CH was partially supported by a grant from the Simons Foundation (245619)}
\thanks{RJLO was supported by an NSF Mathematical Sciences Postdoctoral Research Fellowship}

\newtheorem{theorem}{Theorem}[section]
\newtheorem{lemma}[theorem]{Lemma}
\newtheorem{corollary}[theorem]{Corollary}
\newtheorem{proposition}[theorem]{Proposition}

\newtheorem*{conjecture}{Conjecture}
\newtheorem*{theorem*}{Theorem}
\theoremstyle{remark}\newtheorem*{remark}{Remark}

\numberwithin{equation}{section}

\renewcommand{\pmod}[1]{\left(\mathrm{mod}\,#1\right)}
\renewcommand{\phi}{\varphi}

\begin{document}

\begin{abstract}
The Hardy--Littlewood prime $k$-tuples conjecture has long been thought to be completely unapproachable with current methods.  While this sadly remains true, startling breakthroughs of Zhang, Maynard, and Tao have nevertheless made significant progress toward this problem.  In this work, we extend the Maynard-Tao method to both number fields and the function field $\mathbb{F}_q(t)$.
\end{abstract}

\maketitle

\section{Introduction and statement of results}
\label{sec:introduction}

The classical twin prime conjecture asserts that there are infinitely many primes $p$ such that $p+2$ is also prime.  While this conjecture remains completely out of reach of current methods, there has nevertheless been remarkable recent progress made towards it, beginning with work of Goldston, Pintz, and Y{\i}ld{\i}r{\i}m \cite{GPY2009}, who showed, if $p_n$ denotes the $n$th prime, that
\[
\liminf_{n\to\infty} \frac{p_{n+1}-p_n}{\log p_n} = 0,
\]
so that gaps between consecutive primes can be arbitrarily small when compared with the average gap.  Expanding upon these techniques, Zhang \cite{Zhang2013} proved the amazing result that
\[
\liminf_{n\to\infty} p_{n+1}-p_n \leq 70 \cdot 10^{6},
\]
i.e., that there are bounded gaps between primes!  The techniques of Zhang and Goldston, Pintz, and Y{\i}ld{\i}r{\i}m have subsequently been significantly expanded upon by Maynard \cite{Maynard2013}, Tao, and the Polymath project \cite{Polymath}, so that the best known bound on gaps between primes, at least at the time of writing, is 252.  Remarkably, the techniques of Maynard and Tao also enable one to achieve bounded gaps between $m$ consecutive primes, i.e., that $\liminf (p_{n+m-1}-p_n)$ is finite.

The main idea in all of these results is to attack approximate versions of the \emph{Hardy--Littlewood prime $k$-tuples conjecture}: Given a $k$-tuple $\mathcal{H}=(h_1,\dots,h_k)$ of distinct integers, we say that $\mathcal{H}$ is \emph{admissible} if the set $\{h_1,\dots,h_k\} \bmod p$ is not all of $\mathbb{Z}/p\mathbb{Z}$ for each prime $p$.  The Hardy--Littlewood prime $k$-tuples conjecture can then be stated as follows.

\begin{conjecture}
Given an admissible $k$-tuple $\mathcal{H}=(h_1,\dots,h_k)$, there are infinitely many integers $n$ such that each of $n+h_1,\dots,n+h_k$ is prime.
\end{conjecture}

This conjecture remains intractable at present--- note that the $k=2$ case immediately implies the twin prime conjecture. However, Maynard, Tao, and Zhang have recently succeeded in obtaining partial results that would have seemed incredible just a few years ago.  In particular, we have the following theorem of Maynard \cite{Maynard2013} and Tao.

\begin{theorem*}[Maynard--Tao]
Let $m\geq 2$.  There exists a constant $k_0:=k_0(m)$ such that, for any admissible $k$-tuple $\mathcal{H}=(h_1,\dots,h_k)$ with $k\geq k_0$, there are infinitely many $n$ such that at least $m$ of $n+h_1,\dots,n+h_k$ are prime.
\end{theorem*}

A result on bounded gaps comes from taking $m=2$ and providing an explicit admissible $k_0(2)$-tuple of small diameter.  Indeed, Zhang's \cite{Zhang2013} main theorem is the $m=2$ case of the above, and he obtained $k_0(2)=3.6\cdot 10^{6}$.  Maynard \cite{Maynard2013} was able to take $k_0(2)=105$, and the Polymath project \cite{Polymath} has reduced the permissible value to $k_0(2)=51$.  In our work at hand, we prove an analogue of the Maynard--Tao theorem for number fields and the function field $\mathbb{F}_q(t)$, and we derive corollaries which we believe to be of additional arithmetic interest.

We begin by extending the Maynard--Tao theorem to number fields, for which we must first fix some notation.  Given a number field $K$ with ring of integers $\mathcal{O}_K$, we say that $\alpha\in\mathcal{O}_K$ is \emph{prime} if it generates a principal prime ideal, and we say that a $k$-tuple $(h_1,\dots,h_k)$ of distinct elements of $\mathcal{O}_K$ is \emph{admissible} if the set $\{h_1,\dots,h_k\}$ mod $\mathfrak{p}$ is not all of  $\mathcal{O}_K/\mathfrak{p}$ for each prime ideal $\mathfrak{p}$.  Our first theorem is a direct translation of the Maynard--Tao theorem.

\begin{theorem}
\label{thm:number_field}
Let $m\geq 2$.  There is an integer $k_0:=k_0(m,K)$ such that for any admissible $k$-tuple $(h_1,\dots,h_k)$ in $\mathcal{O}_K$ with $k\geq k_0$, there are infinitely many $\alpha\in\mathcal{O}_K$ such that at least $m$ of $\alpha+h_1,\dots,\alpha+h_k$ are prime.
\end{theorem}

{\noindent \emph{Two remarks:}} \emph{1.} As the proof of Theorem \ref{thm:number_field} will show, the numerology which produces $k_0$ from $m$ is similar to that in Maynard's paper, and is exactly the same if $K$ is totally real.  In general, $k_0$ will depend only upon $m$ and the number of complex embeddings of $K$.

\emph{2.} Another way of extending the Maynard--Tao theorem to number fields was considered by Thorner \cite{Thorner2014}, who proved the analogous result for rational primes satisfying Chebotarev-type conditions (i.e., primes $p$ such that $\mathrm{Frob}_p$ lies in a specified conjugacy-invariant subset of $\mathrm{Gal}(K/\mathbb{Q})$ for some $K/\mathbb{Q}$).
\\

As an immediate corollary to Theorem \ref{thm:number_field}, we can deduce bounded gaps between prime elements of $\mathcal{O}_K$, where the bound depends only on the number of complex embeddings of $K$.  As an example, we have the following corollary for totally real fields. 

\begin{corollary}
\label{cor:bounded_real}
If $K/\mathbb{Q}$ is totally real, then there are infinitely many primes $\alpha_1,\alpha_2\in\mathcal{O}_K$ such that $|\sigma(\alpha_1-\alpha_2)|\leq 600$ for every embedding $\sigma$ of $K$.
\end{corollary}

We now turn our attention to the function field $\mathbb{F}_q(t)$.  Here, the role of primes is played by monic irreducible polynomials in $\mathbb{F}_q[t]$. We define a $k$-tuple $(h_1,\dots,h_k)$ of polynomials in $\mathbb{F}_q[t]$ to be \emph{admissible} if, for each irreducible $P$, the set $\{h_1,\dots,h_k\}$ does not cover all residue classes of $\mathbb{F}_q[t]/P$.

\begin{theorem}
\label{thm:function_field}
Let $m\geq 2$.  There is an integer $k_0:=k_0(m)$, independent of $q$, such that for any admissible $k$-tuple $(h_1,\dots,h_k)$ of polynomials in $\mathbb{F}_q[t]$ with $k\geq k_0$, there are infinitely many $f \in\mathbb{F}_q[t]$ such that at least $m$ of $f+h_1,\dots,f+h_k$ are irreducible.
\end{theorem}

\begin{remark}
Strikingly, the independence of $k_0$ from $q$ passes even so far that Maynard's values of $k_0(m)$ are permissible in this setting as well.  In particular, we may take $k_0(2)=105$.
\end{remark}

As a corollary, we can deduce bounded degree gaps between irreducible polynomials.  In fact, one could already prove something stronger: if $q \geq 3$, then any $a \in \mathbb{F}_q^\times$ occurs infinitely often as a gap (see \cite{Hall2006} for $q > 3$ and \cite{Pollack2008} for $q=3$). These proofs are constructive, but the degrees of the irreducible polynomials produced lie in very sparse sets.  Our next result shows that any $a\in\mathbb{F}_q$, and, indeed, any monomial, in fact occurs in many degrees. Moreover, given any large degree, a positive proportion of elements of $\mathbb{F}_q[t]$ of bounded degree occur as a gap.

\begin{theorem}
\label{thm:functionfield_proportion}
Let $k_0:=k_0(2)$ from Theorem \ref{thm:function_field} and let $q \geq k_0+1$.
\begin{itemize}

\item[(i)] For $d\geq 0$, if $n$ is sufficiently large and satisfies $(n-d,q-1)=1$, each monomial in $\mathbb{F}_q[t]$ of degree $d$ occurs as a gap between monic irreducibles of degree $n$.

\item[(ii)] For $d\geq 0$ and $n$ sufficiently large, the proportion of elements of degree $d$ that appear as gaps between irreducibles in degree $n$ is at least
\[
\frac{1}{k_0-1}-\frac{1}{q-1}.
\]
The same conclusion holds if we restrict to monomials of degree $d$.
\end{itemize}
\end{theorem}

\begin{remark}
The observation that our methods permit us to deduce the first part of Theorem \ref{thm:functionfield_proportion} is due to Alexei Entin.
\end{remark}

This paper is organized as follows.  In Section \ref{sec:method}, we describe the Maynard--Tao method in a general context, and we prove Theorems \ref{thm:number_field} and \ref{thm:function_field} simultaneously.  In Section \ref{sec:applications}, we consider the application of these theorems, and we prove Corollary \ref{cor:bounded_real} and Theorem \ref{thm:functionfield_proportion}.


\section{The general Maynard--Tao method}
\label{sec:method}
 
The Maynard--Tao method for producing primes in tuples is very general, and relies upon a multidimensional variant of the Selberg sieve; indeed, the multidimensional nature of the sieve is the key improvement over the work of Goldston, Pintz, and Y\i ld\i r\i m, and that of Zhang.  Many of the steps in the method are essentially combinatorial, relying principally upon multiplicative functionology and elementary statements, rather than hard information about the structure of the integers or the primes.  It is only in a few key places that deep information is used, and, indeed, these results can be assumed to be ``black boxes''.  As such, when proving our theorems, we proceed in a very general fashion.

We first define general notation and establish a dictionary which permits us to talk simultaneously about the integers (the Maynard--Tao theorem), number fields (Theorem \ref{thm:number_field}), and the function field $\mathbb{F}_q(t)$ (Theorem \ref{thm:function_field}).  This of course introduces some notational obfuscation, but we nevertheless consider this approach useful: first, it enables us to prove each theorem simultaneously, and, second, it elucidates what is needed to prove a Maynard--Tao type result in a general setting.  In Section \ref{subsec:sieve}, we use this dictionary, together with the combinatorial arguments of Maynard, to lay down the proof of the Maynard--Tao theorem, assuming the existence of the relevant black boxes.  It is only in Section \ref{subsec:winning} that we remove ourselves from the general setting and specialize to the number field and function field settings where we have the necessary arithmetic information.  Accordingly, it is here that precise versions of Theorems \ref{thm:number_field} and \ref{thm:function_field} are proved.

\subsection{The dictionary}
\label{subsec:dictionary}

We begin by letting $A$ denote the set of ``integers'' that we are considering.  Thus, in the case of the Maynard--Tao theorem, we will take $A=\mathbb{Z}$. In the number field setting, we will take $A$ to be the ring of integers $\mathcal{O}_K$ of some number field $K/\mathbb{Q}$, and in the function field setting, we will take $A$ to be the polynomial ring $\mathbb{F}_q[t]$.

For any positive integer $N$, we let $A(N)$ denote the ``box of size $N$'' inside $A$.  Over the integers, this is the interval $(N,2N]$. In the polynomial setting, we let $A(N)$ be the collection of monic elements of norm $N$; that is, if $N=q^n$, then $A(N)$ is the set of monic, degree $n$ elements of $\mathbb{F}_q[t]$. The definition of $A(N)$ is slightly more complicated in the number field situation. We first define $A_0(N)$ as the set of $\alpha \in \mathcal{O}_K$ which satisfy $0 < \sigma(\alpha) \leq N$ for all real embeddings $\sigma\colon K\hookrightarrow \mathbb{C}$ and satisfy $|\sigma(\alpha)| \leq N$ for all complex embeddings. We then take $A(N) := A_0(2N) \setminus A_0(N)$.

Given a nonzero ideal $\mathfrak{q}\subseteq A$, we define analogues of three classical multiplicative functions, namely the norm $|\mathfrak{q}|:=|A/\mathfrak{q}|$, the ``phi-function'' $\varphi(\mathfrak{q}) :=|(A/\mathfrak{q})^\times|$, and the M\"obius function $\mu(\mathfrak{q}):=(-1)^r$ if $\mathfrak{q}=\mathfrak{p}_1\dots\mathfrak{p}_r$ for distinct prime ideals $\mathfrak{p}_1,\dots,\mathfrak{p}_r$ and $\mu(\mathfrak{q})=0$ otherwise. We define the zeta function of $A$ by
\[
\zeta_A(s) := \sum_{\mathfrak{q}\subseteq A} |\mathfrak{q}|^{-s}.
\]
When $A = \mathcal{O}_K$, the function $\zeta_A(s)$ is the usual Dedekind zeta function of $K$. When $A= \mathbb{F}_q[t]$, one has the closed form expression $\zeta_A(s) = \frac{1}{1-q^{1-s}}$. This differs from the usual zeta function of $\mathbb{F}_q(t)$ in that the Euler factor corresponding to the prime over $1/t$ has been removed.

We record here that the number of elements $\alpha\in A(N)$ satisfying a congruence condition $\alpha \equiv \alpha_0 \pmod{\mathfrak{q}}$ is given by
\[
\frac{|A(N)|}{|\mathfrak{q}|} + O(|\partial A(N,\mathfrak{q})|),
\]
where
\[
|\partial A(N,\mathfrak{q})| \ll
\begin{cases}
1 &\text{if $A=\mathbb{Z}$ or $A=\mathbb{F}_q[t]$}, \\
1+(\frac{|A(N)|}{|\mathfrak{q}|})^{1-\frac{1}{d}} &\text{if $A=\mathcal{O}_K$ and $[K:\mathbb{Q}]=d$}.
\end{cases}
\]
In fact, if $A=\mathbb{F}_q[t]$ and $|A(N)| \geq |\mathfrak{q}|$, then we can take $|\partial A(N,\mathfrak{q})|=0$. When $A=\mathbb{Z}$ or $A=\mathbb{F}_q[t]$, these estimates for $|\partial A(N,\mathfrak{q})|$ are trivial. When $A=\mathcal{O}_K$, matters are more complicated but still relatively familiar. One starts by embedding $K$ into Minkowski space $\mathbb{R}^{r_1} \times \mathbb{C}^{r_2}$. Under this embedding, $\mathfrak{q}$ goes to a lattice, while the constraints on $A(N)$ correspond to a certain region of $\mathbb{R}^{r_1} \times \mathbb{C}^{r_2}$. The estimate for $|\partial A(N,\mathfrak{q})|$ comes from estimating the number of translates of the fundamental parallelogram that intersect the boundary of that region. (Compare with the proof of \cite[Lemma 1]{Kaptan2013}.) For our purposes, what is important to take away is that we always have a power savings in the error term: $|\partial A(N,\mathfrak{q})| \ll (|A(N)|/|\mathfrak{q}|)^{1-\nu}$ for some positive $\nu = \nu(A)$, as long as $|A(N)| \geq |\mathfrak{q}|$.



Let $P$ denote the ``prime'' elements of $A$ and take $P(N) = P \cap A(N)$.  If $A=\mathbb{Z}$, $P$ is simply the set of primes, and, if $A=\mathcal{O}_K$, $P$ is the set of generators of principal prime ideals.  If $A=\mathbb{F}_q[t]$, $P$ is the set of monic irreducible polynomials.  In all of these cases, we have a prime number theorem of the form
\[
|P(N)| \sim c \cdot \frac{|A(N)|}{\log N}
\]
for some constant $c$, and we moreover have a prime number theorem for the set $P(N;\mathfrak{q},\alpha_0)$ of primes in the coprime residue class $\alpha_0\pmod{\mathfrak{q}}$ of the form
\[
|P(N;\mathfrak{q},\alpha_0)| = \frac{1}{\varphi(\mathfrak{q})} |P(N)| + \mathcal{E}(N;\mathfrak{q},\alpha_0).
\]
For any individual $\mathfrak{q}$, we have the upper bound $\mathcal{E}(N;\mathfrak{q},\alpha_0) = o_\mathfrak{q}(P(N))$, and we say that $P$ has \emph{level of distribution} $\theta>0$ if, for any $B>0$, the bound
\[
\sum_{|\mathfrak{q}|\leq Q} \max_{\substack{\alpha_0\pmod{\mathfrak{q}} \\ (\alpha_0,\mathfrak{q})=1}} |\mathcal{E}(N;\mathfrak{q},\alpha_0)| \ll_B \frac{|A(N)|}{\log^B N}
\]
holds for all $Q\leq |A(N)|^\theta$ and all sufficiently large $N$.  If $A = \mathbb{Z}$, the Bombieri-Vinogradov theorem asserts that the primes have level of distribution $\theta$ for any $\theta < 1/2$ and the Elliott-Halberstam conjecture is that any $\theta<1$ is permissible (see \cite{FI2010} for more information).   A generalized form of the Bombieri-Vinogradov theorem due to Hinz \cite{Hinz1988} shows that the primes in $\mathcal{O}_K$ have some level of distribution $\theta$, the specific value depending only on the number of complex conjugate embeddings of $K$; in particular, any totally real field has level of distribution $\theta$ for any $\theta<1/2$.  Finally, in the function field setting, Weil's proof of the Riemann hypothesis for curves implies that we may take any $\theta<1/2$.

\subsection{Sieve manipulations: Multiplicative functionology}
\label{subsec:sieve}

We are now ready to describe the Maynard--Tao method in general terms; our exposition follows that of Maynard \cite{Maynard2013}, to which we make frequent reference.  We say that a tuple $h_1,\dots,h_k\in A$ is admissible if it does not cover all residue classes modulo $\mathfrak{p}$ for any prime ideal $\mathfrak{p}$ of $A$.  The main objects of consideration are the sums
\[
S_1 := \sum_{\begin{subarray}{c} \alpha \in A(N) \\ \alpha \equiv v_0\pmod{\mathfrak{w}} \end{subarray}} \left(\sum_{\begin{subarray}{c} \mathfrak{d}_1,\dots,\mathfrak{d}_k: \\ \mathfrak{d_i} \mid (\alpha+h_i)\,\forall i \end{subarray}} \lambda_{\mathfrak{d}_1,\dots,\mathfrak{d}_k} \right)^2
\]
and
\[
S_2:= \sum_{\begin{subarray}{c} \alpha \in A(N) \\ \alpha \equiv v_0\pmod{\mathfrak{w}} \end{subarray}} \left( \sum_{i=1}^k \chi_P(\alpha+h_i) \right) \left(\sum_{\begin{subarray}{c} \mathfrak{d}_1,\dots,\mathfrak{d}_k: \\ \mathfrak{d_i} \mid (\alpha+h_i) \,\forall i \end{subarray}} \lambda_{\mathfrak{d}_1,\dots,\mathfrak{d}_k} \right)^2,
\]
where $\chi_P(\cdot)$ denotes the characteristic function of $P$, $\lambda_{\mathfrak{d}_1,\dots,\mathfrak{d}_k}$ are suitably chosen weights, $\mathfrak{w}:=\prod_{|\mathfrak{p}|<D_0}\mathfrak{p}$ for some $D_0$ tending slowly to infinity with $N$, say $D_0=\log\log\log N$, and $v_0$ is a residue class modulo $\mathfrak{w}$ chosen so that each $\alpha+h_i$ lies in $A/\mathfrak{w}^\times$.

Because each summand is non-negative, if we can show that $S_2 > \rho S_1$ for some positive $\rho$, then there must be at least one $\alpha\in A(N)$ for which more than $\rho$ of the values $\alpha+h_1,\dots,\alpha+h_k$ are prime.  This is our goal, and it is where the art of choosing the weights $\lambda_{\mathfrak{d}_1,\dots,\mathfrak{d}_k}$ comes into play.  We begin by making some assumptions regarding their support.  In particular, given $\mathfrak{d}_1,\dots,\mathfrak{d}_k$, define $\mathfrak{d}:=\prod_{i=1}^k \mathfrak{d}_i$, and set $\lambda_{\mathfrak{d}_1,\dots,\mathfrak{d}_k}=0$ unless $(\mathfrak{d},\mathfrak{w})=1$, $\mathfrak{d}$ is squarefree, and $|\mathfrak{d}|\leq R$, where $R$ will be chosen later to be a small power of $|A(N)|$.  The main result of this section is the following.

\begin{proposition}
\label{prop:main}
Suppose that the primes $P$ have level of distribution $\theta>0$, and set $R=|A(N)|^{\theta/2 - \delta}$ for some small $\delta>0$.  Given a piecewise differentiable function $F\colon [0,1]^k \to \mathbb{R}$ supported on the simplex $\mathcal{R}_k:=\{(x_1,\dots,x_k)\in[0,1]^k: x_1+\dots+x_k \leq 1\}$, let $F_{\mathrm{max}}:=\sup_{(t_1, \dots, t_k) \in [0,1]^k} |F(t_1, \dots, t_k)|
+
\sum_{i=1}^k \left|\frac{\partial F}{\partial x_i}(t_1, \dots, t_k)  \right|$.  If we set
\[
\lambda_{\mathfrak{d}_1,\dots,\mathfrak{d}_k} := \left(\prod_{i=1}^k \mu(\mathfrak{d}_i)|\mathfrak{d}_i|\right) \sum_{\substack{\mathfrak{r}_1,\dots,\mathfrak{r}_k \\ \mathfrak{d}_i\mid\mathfrak{r}_i\,\forall i \\ (\mathfrak{r}_i,\mathfrak{w})=1 \,\forall i }} \frac{\mu(\mathfrak{r}_1\dots\mathfrak{r}_k)^2}{\prod_{i=1}^k \varphi(\mathfrak{r_i})} F\left( \frac{\log|\mathfrak{r}_1|}{\log R},\dots, \frac{\log |\mathfrak{r}_k|}{\log R}\right)
\]
whenever $|\mathfrak{d}_1\dots\mathfrak{d}_k|<R$ and $(\mathfrak{d}_1\dots\mathfrak{d}_k,\mathfrak{w})=1$, and $\lambda_{\mathfrak{d}_1,\dots,\mathfrak{d}_k}=0$ otherwise, then
\[
S_1 = \frac{(1+o(1)) \varphi(\mathfrak{w})^k |A(N)| (c_A \log R)^k}{|\mathfrak{w}|^{k+1}}I_k(F)
\]
and
\[
S_2 = \frac{(1+o(1)) \varphi(\mathfrak{w})^k |P(N)| (c_A \log R)^{k+1}}{|\mathfrak{w}|^{k+1}} \sum_{m=1}^k J_k^{(m)}(F)
\]
where $c_A$ is the residue at $s=1$ of $\zeta_A(s)$,
\[
I_k(F) := \idotsint_{\mathcal{R}_k} F(x_1,\dots,x_k)^2 \,dx_1 \dots dx_k
\]
and
\[
J_k^{(m)}(F) := \idotsint_{[0,1]^{k-1}} \left( \int_0^1 F(x_1,\dots,x_k) \, dx_m\right)^2 dx_1 \dots dx_{m-1}dx_{m+1}\dots dx_k.
\]
\end{proposition}

Before we can prove Proposition \ref{prop:main}, we first show that, by diagonalizing the quadratic form, we can rewrite $S_1$ and $S_2$.  We begin with $S_1$.

\begin{lemma} \label{lem:maynard_first_sum}
For ideals $\mathfrak{r}_1,\dots,\mathfrak{r}_k$, let
\[
y_{\mathfrak{r}_1, \dots, \mathfrak{r}_k} =
\left(
  \prod_{i=1}^k \mu(\mathfrak{r}_i) \varphi(\mathfrak{r}_i)
\right)
\sum_{\substack{\mathfrak{d}_1, \dots, \mathfrak{d}_k \\ \mathfrak{r}_i \mid \mathfrak{d}_i\,\forall i}}
  \frac{\lambda_{\mathfrak{d}_1,  \dots, \mathfrak{d}_k}}{\prod_{i=1}^k |\mathfrak{d}_i|},
\]
and set $y_{\max} = \sup_{\mathfrak{r}_1, \dots, \mathfrak{r}_k} |y_{\mathfrak{r}_1, \dots, \mathfrak{r}_k}|$. If $R = |A(N)|^{1/2-\delta}$ for some $\delta>0$, then
\[
S_1 = \frac{|A(N)|}{|\mathfrak{w}|}
\sum_{\mathfrak{r}_1, \dots, \mathfrak{r}_k} \frac{y^2_{\mathfrak{r}_1, \dots, \mathfrak{r}_k}}{\prod_{i=1}^k \varphi(\mathfrak{r}_i)}
+ O \left(
  \frac{y^2_{\max} |A(N)| \phi(\mathfrak{w})^k (\log R)^k}{|\mathfrak{w}|^{k+1} D_0}
\right).
\]
\end{lemma}

\begin{remark}
The change of variables to $y_{\mathfrak{r}_1,\dots,\mathfrak{r}_k}$ is invertible, the proof of which relies only on elementary manipulations  (see \cite[p. 9]{Maynard2013}).
\end{remark}

\begin{proof}[Proof of Lemma \ref{lem:maynard_first_sum}]
We begin by expanding the square and interchanging the order of summation to obtain
\begin{align*}
S_1
	&= \sum_{\mathfrak{d}_1,\dots,\mathfrak{d}_k} \sum_{\mathfrak{e}_1,\dots,\mathfrak{e}_k} \lambda_{\mathfrak{d}_1,\dots,\mathfrak{d}_k} \lambda_{\mathfrak{e}_1,\dots,\mathfrak{e}_k} \sum_{\substack{ \alpha\in A(N) \\ \alpha\equiv v_0\pmod{\mathfrak{w}} \\ \alpha\equiv -h_i \pmod{[\mathfrak{d}_i,\mathfrak{e}_i]}\,\forall i}} 1 \\
&= \frac{|A(N)|}{|\mathfrak{w}|} \sideset{}{^\prime}\sum_{\substack{\mathfrak{d}_1,\dots,\mathfrak{d}_k \\ \mathfrak{e}_1,\dots,\mathfrak{e}_k}} \frac{\lambda_{\mathfrak{d}_1,\dots,\mathfrak{d}_k} \lambda_{\mathfrak{e}_1,\dots,\mathfrak{e}_k}}{\prod_{i=1}^k|[\mathfrak{d}_i,\mathfrak{e}_i]|} + O\left( \lambda_{\max}^2 \sideset{}{^\prime}\sum_{\substack{\mathfrak{d}_1,\dots,\mathfrak{d}_k \\ \mathfrak{e}_1,\dots,\mathfrak{e}_k}} |\partial A(N,\mathfrak{w}\prod_{i=1}^{k}[\mathfrak{d}_i,\mathfrak{e}_i])| \right).
\end{align*}
Here the $\prime$ on the summation indicates it is to be taken over those $\mathfrak{d}_1,\dots,\mathfrak{d}_k,\mathfrak{e}_1,\dots,\mathfrak{e}_k$ for which the congruence conditions modulo $\mathfrak{w}, [\mathfrak{d}_1,\mathfrak{e}_1],\dots,[\mathfrak{d}_k,\mathfrak{e}_k]$ admit a simultaneous solution; note that in that case, $\mathfrak{w}, [\mathfrak{d}_1,\mathfrak{e}_1],\dots,[\mathfrak{d}_k,\mathfrak{e}_k]$ are pairwise coprime. Now recall from \S\ref{subsec:dictionary} that $|\partial A(N,\mathfrak{q})| \ll (|A(N)|/|\mathfrak{q}|)^{1-\nu}$ for some $\nu > 0$, provided that $|A(N)| \geq |\mathfrak{q}|$. This implies that the above error is
\[ \ll \lambda_{\max}^2  \cdot |A(N)|^{1-\nu} \sideset{}{^\prime}\sum_{\substack{\mathfrak{d}_1,\dots,\mathfrak{d}_k \\ \mathfrak{e}_1,\dots,\mathfrak{e}_k}} \frac{1}{\prod_{i=1}^{k} |[\mathfrak{d}_i, \mathfrak{e}_i]|^{1-\nu}}. \]
For each $\mathfrak{q}$, the number of ways of choosing $\mathfrak{d}_1, \dots, \mathfrak{d}_k$ and $\mathfrak{e}_1, \dots, \mathfrak{e}_k$ so that $\prod_{i=1}^{k} [\mathfrak{d}_i, \mathfrak{e}_i] = \mathfrak{q}$ is at most $\tau_{3k}(\mathfrak{q})$. Hence our error is
\begin{multline*} \ll \lambda_{\max}^2 |A(N)|^{1-\nu} \cdot \sum_{|\mathfrak{q}| \leq R^2} \frac{\mu^2(\mathfrak{q})\tau_{3k}(\mathfrak{q})}{|\mathfrak{q}|^{1-\nu}} \leq \lambda_{\max}^2 |A(N)|^{1-\nu} \cdot R^{2\nu} \prod_{|\mathfrak{p}| \leq R^2} \left(1+\frac{3k}{|\mathfrak{p}|}\right) \\ \ll \lambda_{\max}^2 |A(N)|^{1-\nu} \cdot R^{2\nu} (\log{R})^{3k} = \lambda_{\max}^2 |A(N)| (\log{R})^{3k} \cdot (|A(N)|/R^2)^{-\nu}. \end{multline*}
(To go from the first line to the second, we used a version of Mertens' theorem for global fields. See, for example, \cite{Rosen99}. This sort of estimation of sums by Euler products will be used frequently in what follows without further comment.)
Because $R = |A(N)|^{1/2-\delta}$, this error is negligible compared to the error claimed in the statement of the lemma.

We now focus our attention on the main term. Following Maynard's manipulations to uncouple the interdependence of $\mathfrak{d}_i$ and $\mathfrak{e}_j$ and making the change of variables indicated in the statement of the lemma, the main term becomes
\[
\frac{|A(N)|}{|\mathfrak{w}|}
	\sum_{\mathfrak{u}_1,\dots,\mathfrak{u}_k} \left(\prod_{i=1}^k\frac{\mu(\mathfrak{u}_i)^2}{\varphi(\mathfrak{u}_i)}\right)
	\sideset{}{^*}\sum_{\mathfrak{s}_{1,2},\dots,\mathfrak{s}_{k,k-1}} \left(\prod_{\begin{subarray}{c}1\leq i,j\leq k \\ i\neq j \end{subarray}} \frac{\mu(\mathfrak{s}_{i,j})}{\varphi(\mathfrak{s}_{i,j})^2}\right)
	y_{\mathfrak{a}_1,\dots,\mathfrak{a}_k}y_{\mathfrak{b}_1,\dots,\mathfrak{b}_k},
\]
and we note, for consideration of the error term, that $\lambda_{\max} \ll y_{\max} \log^k R$.  In the above, $\mathfrak{a}_i:=\mathfrak{u}_i\prod_{j\neq i}\mathfrak{s}_{i,j}$, $\mathfrak{b}_j:=\mathfrak{u}_j\prod_{i\neq j} \mathfrak{s}_{i,j}$, and the $*$ on the summation indicates it is to be taken over $\mathfrak{s}_{i,j}$ such that $(\mathfrak{s}_{i,j},\mathfrak{u}_i)=(\mathfrak{s}_{i,j},\mathfrak{u}_j)=1=(\mathfrak{s}_{i,j},\mathfrak{s}_{a,j})=(\mathfrak{s}_{i,j},\mathfrak{s}_{i,b})$ for all $a\neq i$, $b\neq j$.  Moreover, considering the support of the $y$'s, if some $\mathfrak{s}_{i,j} \neq 1$, then $|\mathfrak{s}_{i,j}|>D_0$ owing to the fact that $(\mathfrak{s}_{i,j},\mathfrak{w})=1$.  The contribution in that case is at most
\begin{align}\label{eq:S1finalestimate}
\frac{y_{\max}^2|A(N)|}{|\mathfrak{w}|}
	\left( \sum_{\begin{subarray}{c}|\mathfrak{u}|\leq R \\ (\mathfrak{u},\mathfrak{w})=1 \end{subarray}} \frac{\mu(\mathfrak{u})^2}{\varphi(\mathfrak{u})}\right)^k \!\!\!\!
	\left(\sum_{|\mathfrak{s}_{i,j}|>D_0} \frac{\mu(\mathfrak{s}_{i,j})^2}{\varphi(\mathfrak{s}_{i,j})^2} \right) \!\!
	\left( \sum_{\mathfrak{s}\subseteq A} \frac{\mu(\mathfrak{s})^2}{\varphi(\mathfrak{s})^2} \right)^{k^2-k-1}
\!\!\!\!\!\!\!\!\ll \!\!
	\frac{y_{\max}^2|A(N)| \varphi(\mathfrak{w})^k \log^k R}{|\mathfrak{w}|^{k+1} D_0}.
\end{align}
We may thus restrict our attention only to those terms arising from $\mathfrak{s}_{i,j}=1$ for all $i\neq j$.  The lemma follows.
\end{proof}

We now turn to $S_2$, the handling of which will require more delicate information than was needed for $S_1$.  We first define, for $1 \leq m \leq k$, the component sums
\[
S_2^{(m)} := \sum_{\begin{subarray}{c} \alpha \in A(N) \\ \alpha \equiv v_0\pmod{\mathfrak{w}} \end{subarray}} \chi_P(\alpha+h_m) \left(\sum_{\begin{subarray}{c} \mathfrak{d}_1,\dots,\mathfrak{d}_k: \\ \mathfrak{d_i} \mid (\alpha+h_i)\,\forall i \end{subarray}} \lambda_{\mathfrak{d}_1,\dots,\mathfrak{d}_k} \right)^2,
\]
so that $S_2 = \sum_{m=1}^k S_2^{(m)}$. To rewrite $S_2^{(m)}$ in a manner similar to what was done with $S_1$, we need information about how the primes are distributed in arithmetic progressions.  Specifically, we will need the assumption that $P$ has level of distribution $\theta>0$.

\begin{lemma} \label{lem:maynard_second_sum}
Assume that $P$ has level of distribution $\theta > 0$ and that $R = |A(N)|^{\theta/2 - \varepsilon}.$ Let
\[
y^{(m)}_{\mathfrak{r}_1, \dots, \mathfrak{r}_k} =
\left(
  \prod_{i=1}^k \mu(\mathfrak{r}_i) g(\mathfrak{r}_i)
\right)
\sum_{\substack{\mathfrak{d}_1, \dots, \mathfrak{d}_k \\ \mathfrak{r}_1 | \mathfrak{d}_i\,\forall i \\ \mathfrak{d}_m = 1}}
\frac{\lambda_{\mathfrak{d}_1, \dots, \mathfrak{d}_k}}{\prod_{i=1}^k \varphi(\mathfrak{d}_i)},
\]
where $g$ is the multiplicative function defined by $g(\mathfrak{p}) = |\mathfrak{p}|-2$ for all prime ideals $\mathfrak{p}$ of $A$.
Let $y^{(m)}_{max} = \sup_{\mathfrak{r}_1, \dots, \mathfrak{r}_k} |y^{(m)}_{\mathfrak{r}_1, \dots, \mathfrak{r}_k}|$. Then, for any fixed $B>0$ we have
\begin{align*}
S_2^{(m)} & =
\frac{|P(N)|}{\varphi(\mathfrak{w})}
\sum_{\mathfrak{r}_1, \dots, \mathfrak{r}_k}\frac{(y^{(m)}_{\mathfrak{r}_1, \dots, \mathfrak{r}_k})^2}{\prod_{i=1}^k g(\mathfrak{r}_i)} \\
&+ O\left(
  \frac{(y^{(m)}_{max})^2 \varphi(\mathfrak{w})^{k-2} |A(N)| (\log N)^{k-2}}{|\mathfrak{w}|^{k-1} D_0}
\right)
+ O\left(
\frac{y^2_{max} |A(N)|}{(\log N)^B}
\right).
\end{align*}
\end{lemma}
\begin{proof}

We begin by expanding out the square and swapping the order of summation, obtaining
$$
S_2^{(m)} = \sum_{\substack{\mathfrak{d}_1,\dots,\mathfrak{d}_k \\ \mathfrak{e}_1,\dots, \mathfrak{e}_k}} \lambda_{\mathfrak{d}_1,\dots,\mathfrak{d}_k} \lambda_{\mathfrak{e}_1,\dots,\mathfrak{e}_k} \sum_{\substack{\alpha \in A(N) \\ \alpha \equiv v_0 \pmod{\mathfrak{w}} \\ [\mathfrak{d}_i, \mathfrak{e}_i] \mid (\alpha + h_i)\,\forall i}} \chi_P(\alpha + h_m).
$$

As in Lemma \ref{lem:maynard_first_sum}, we rewrite the inner sum over a single residue class modulo $\mathfrak{q} = \mathfrak{w} \prod_{i=1}^k [\mathfrak{d}_i, \mathfrak{e}_i]$, which we may do if the ideals $\mathfrak{w}, [\mathfrak{d}_1, \mathfrak{e}_1],\dots, [\mathfrak{d}_k, \mathfrak{e}_k]$ are pairwise coprime.  The element $\alpha + h_m$ will lie in a residue class coprime to the modulus if and only if $\mathfrak{d}_m = \mathfrak{e}_m = 1$, the trivial ideal.  Based on our choice of $v_0\pmod{\mathfrak{w}}$, this is the only case that yields a contribution.  We find that
$$
\sum_{\substack{\alpha \in A(N) \\ \alpha \equiv v_0 \pmod{\mathfrak{w}} \\ [\mathfrak{d}_i, \mathfrak{e}_i] \mid (\alpha + h_i)\,\forall i}} \chi_P(\alpha + h_m) = \frac{|P(N)|}{\varphi(\mathfrak{q})} + O\left(\left(\frac{|A(N)|}{|\mathfrak{w}|\prod_{i=1}^{k}|[\mathfrak{d}_i, \mathfrak{e}_i]|}\right)^{1-\nu}\right)  +O\left(\mathcal{E}(N;\mathfrak{q},\alpha_0)\right),
$$
where we recall that $P(N)=P \cap A(N)$. (The first $O$-term is needed in the number field case, since it is $\alpha$ that is restricted to $A(N)$ instead of $\alpha+h_m$.) Letting $\mathcal{E}(N;\mathfrak{q}):=\max_{(\alpha_0,\mathfrak{q})=1} |\mathcal{E}(N;\mathfrak{q},\alpha_0)|$, we thus find that
\begin{multline}
S_2^{(m)} = \frac{|P(N)|}{\varphi(\mathfrak{w})} \sum_{\substack{\mathfrak{d}_1,\dots,\mathfrak{d}_k \\ \mathfrak{e}_1,\dots,\mathfrak{e}_k\\ \mathfrak{e}_m = \mathfrak{d}_m = 1}} \frac{\lambda_{\mathfrak{d}_1,\dots,\mathfrak{d}_k} \lambda_{\mathfrak{e}_1,\dots,\mathfrak{e}_k}}{\prod_{i=1}^k \varphi([\mathfrak{d}_i, \mathfrak{e}_i])} +
 O\left(\sum_{\substack{\mathfrak{d}_1,\dots,\mathfrak{d}_k \\ \mathfrak{e}_1,\dots,\mathfrak{e}_k}} |\lambda_{\mathfrak{d}_1,\dots,\mathfrak{d}_k} \lambda_{\mathfrak{e}_1,\dots,\mathfrak{e}_k}| \mathcal{E}(N;\mathfrak{q})\right) \\
 + O\left(\lambda_{\max}^2 \cdot |A(N)|^{1-\nu}\sum_{\substack{\mathfrak{d}_1,\dots,\mathfrak{d}_k \\ \mathfrak{e}_1,\dots,\mathfrak{e}_k}}  \frac{1}{\prod_{i=1}^{k} |[\mathfrak{d}_i, \mathfrak{e}_i]|^{1-\nu}}\right).\label{eq:S2mainerror}
\end{multline}
The second error term is $O(\lambda_{\max}^2 |A(N)| (\log{R})^{3k} \cdot (|A(N)|/R^2)^{-\nu})$, by an argument already appearing in the proof of Lemma \ref{lem:maynard_first_sum}. This is negligible for us. Now consider the first $O$-term.  For any $\mathfrak{q}$, there are at most $\tau_{3k}(\mathfrak{q})$ ways to choose $k$-tuples $\mathfrak{d}_1,\dots,\mathfrak{d}_k, \mathfrak{e}_1,\dots,\mathfrak{e}_k$ such that $\mathfrak{w} \prod_{i=1}^k [\mathfrak{d}_i, \mathfrak{e}_i] = \mathfrak{q}$. Recall that $\lambda_{\max} \ll y_{\max} \log^k R$.  Moreover, by our assumptions on the support of $\lambda_{\mathfrak{d}_1,\dots,\mathfrak{d}_k}$, the modulus $\mathfrak{q}$ satisfies $|\mathfrak{q}| \leq |\mathfrak{w}|R^2$.  Thus, the error term in \eqref{eq:S2mainerror} contributes no more than
\begin{align}
\label{eq:S2errorbd1}
y_{\max}^2 (\log R)^{2k} \sum_{|\mathfrak{q}| < R^2|\mathfrak{w}|} \mu(\mathfrak{q})^2 \tau_{3k}(\mathfrak{q})\mathcal{E}(N;\mathfrak{q}).
\end{align}
We now recall that we have assumed that the primes $P$ have level of distribution $\theta$, and we have taken $R=|A(N)|^{\theta/2-\epsilon}$.  Using the trivial bound $\mathcal{E}(N;\mathfrak{q}) \ll |A(N)|/\varphi(\mathfrak{q})$ along with the Cauchy-Schwarz inequality, we therefore find that
\begin{eqnarray*}
\sum_{|\mathfrak{q}| < R^2|\mathfrak{w}|} \mu(\mathfrak{q})^2 \tau_{3k}(\mathfrak{q})\mathcal{E}(N;\mathfrak{q})
\!\!\!	&\ll & \!\!\! \left(\sum_{|\mathfrak{q}|<R^2|\mathfrak{w}|} \mu(\mathfrak{q})^2 \tau_{3k}^2(\mathfrak{q}) \frac{|A(N)|}{\varphi(\mathfrak{q})}\right)^{1/2}\left(\sum_{|\mathfrak{q}| < R^2|\mathfrak{w}|} \mu(\mathfrak{q})^2 \mathcal{E}(N; \mathfrak{q})\right)^{1/2} \\
\!\!\!	&\ll & \!\!\! \frac{|A(N)|}{(\log N)^B}
\end{eqnarray*}
for any large $B$.

Now that we have handled the error term, we are free to concentrate on the main term. As in the proof of Lemma \ref{lem:maynard_first_sum}, we decouple $\mathfrak{d}_i$ and $\mathfrak{e}_j$ by introducing an auxilliary summation over ideals $\mathfrak{s}_{i,j}$, and we define the function multiplicative function $g(\mathfrak{a})$ by $g(\mathfrak{p})=|\mathfrak{p}|-2$, so that
$$
\frac{1}{\varphi([\mathfrak{d}_i, \mathfrak{e}_i])} = \frac{1}{\varphi(\mathfrak{d}_i) \varphi(\mathfrak{e}_i)} \sum_{\mathfrak{u}_i \mid \mathfrak{d}_i \mathfrak{e}_i} g(\mathfrak{u}_i).
$$
Our main term can thus be written as
\begin{align}
\label{eq:S2mainterm1}
\frac{|P(N)|}{\varphi(\mathfrak{w})} \sum_{\substack{\mathfrak{u}_1,\dots,\mathfrak{u}_k \\ \mathfrak{u}_m=1}} \left(\prod_{i=1}^k g(\mathfrak{u}_i)\right) \sideset{}{^*}\sum_{\mathfrak{s}_{1,2},\dots,\mathfrak{s}_{k, k-1}} \left(\prod_{1 \leq i, j \leq k} \mu(\mathfrak{s}_{i, j})\right) \sum_{\substack{\mathfrak{d}_1,\dots, \mathfrak{d}_k \\ \mathfrak{e}_1,\dots, \mathfrak{e}_k \\ \mathfrak{u}_i \mid \mathfrak{d}_i \mathfrak{e}_i\,\forall i \\ \mathfrak{s}_{i, j} \mid \mathfrak{d}_i \mathfrak{e}_j\,\forall i \neq j \\ \mathfrak{d}_m = \mathfrak{e}_m = 1}} \frac{\lambda_{\mathfrak{d}_1,\dots,\mathfrak{d}_k} \lambda_{\mathfrak{e}_1,\dots,\mathfrak{e}_k}}{\prod_{i=1}^k \varphi(\mathfrak{d}_i) \varphi(\mathfrak{e}_i)}.
\end{align}
We now make the change of variables indicated in the statement of the lemma.  This yields
$$
\frac{|P(N)|}{\varphi(\mathfrak{w})} \sum_{\substack{\mathfrak{u}_1,\dots,\mathfrak{u}_k \\ \mathfrak{u}_m=1}} \left(\prod_{i=1}^k \frac{\mu(\mathfrak{u}_i)^2}{g(\mathfrak{u}_i)}\right) \sum_{\mathfrak{s}_{1,2},\dots,\mathfrak{s}_{k, k-1}} \left(\prod_{\substack{1 \leq i, j \leq k \\ i \neq j}} \frac{\mu(\mathfrak{s}_{i, j})}{g(\mathfrak{s}_{i, j})^2}\right) y_{\mathfrak{a}_1,\dots,\mathfrak{a}_k}^{(m)} y_{\mathfrak{b}_1,\dots,\mathfrak{b}_k}^{(m)},
$$
where the $\mathfrak{a}_i$'s and $\mathfrak{b}_j$'s are defined as in the proof of Lemma \ref{lem:maynard_first_sum}. When some $\mathfrak{s}_{i, j} \neq 1$, the contribution is
\begin{align*}
& \ll \frac{(y_{\mathrm{max}}^{(m)})^2 |A(N)|}{\varphi(\mathfrak{w}) \log N}\left(\sum_{\substack{ |\mathfrak{u}| < R \\ (\mathfrak{u}, \mathfrak{w}) = 1}} \frac{\mu(\mathfrak{u})^2}{g(\mathfrak{u})}\right)^{k-1} \left(\sum_{\mathfrak{s}} \frac{\mu(\mathfrak{s})^2}{g(\mathfrak{s})^2}\right)^{k(k-1)-1} \sum_{|\mathfrak{s}_{i, j}| > D_0} \frac{\mu(\mathfrak{s}_{i, j})^2}{g(\mathfrak{s}_{i, j})^2} \\
& \ll \frac{(y_{\max}^{(m)})^2 \varphi(\mathfrak{w})^{k-2} |A(N)| (\log R)^{k-1}}{|\mathfrak{w}|^{k-1}D_0 \log N}.
\end{align*}
Putting all of this together, we find that
$$
S_2^{(m)} = \frac{|P(N)|}{\varphi(\mathfrak{w})} \sum_{\mathfrak{u}_1,\dots,\mathfrak{u}_k} \frac{(y_{\mathfrak{u}_1,\dots,\mathfrak{u}_k}^{(m)})^2}{\prod_{i=1}^k g(\mathfrak{u}_i)} + O\left(\frac{(y_{\max}^{(m)})^2 \varphi(\mathfrak{w})^{k-2} |A(N)| (\log R)^{k-2}}{D_0 |\mathfrak{w}|^{k-1}} + \frac{y_{\max}^2 |A(N)|}{(\log N)^B}\right),
$$
as claimed.
\end{proof}

We note that the quantities $y^{(m)}_{\mathfrak{r}_1,\dots,\mathfrak{r}_k}$ can be related to the variables $y_{\mathfrak{r}_1,\dots,\mathfrak{r}_k}$.

\begin{lemma}
\label{lem:y^{(m)}}
If $\mathfrak{r}_m=1$, then
\[
y^{(m)}_{\mathfrak{r}_1,\dots,\mathfrak{r}_k} = \sum_{\mathfrak{a}_m} \frac{y_{\mathfrak{r}_1,\dots,\mathfrak{r}_{m-1},\mathfrak{a}_{m},\mathfrak{r}_{m+1},\dots,\mathfrak{r}_k}}{\varphi(\mathfrak{a}_m)} + O\left(\frac{y_{\mathrm{max}}\varphi(\mathfrak{w}) \log R}{|\mathfrak{w}|D_0}\right).
\]
\end{lemma}
\begin{proof}
The proof of this result relies upon combinatorial manipulations and standard estimates, and, using the ideas in Lemmas \ref{lem:maynard_first_sum} and \ref{lem:maynard_second_sum} can be deduced almost \emph{mutatis mutandis} from Maynard's proof of Lemma 5.3 \cite{Maynard2013}.
\end{proof}

We are now ready to make a specific choice of our sieve weights.  In particular, by choosing $y_{r_1,\dots,r_k}$ to be determined by the values of a smooth function, we will be able to express $S_1$ and $S_2^{(m)}$ in particularly nice terms.  Thus, let $F\colon [0,1]^k \to \mathbb{R}$ be a piecewise differentiable function supported on $\left\{(x_1, \dots, x_k) \in [0,1]^k: \sum_{i=1}^k x_i \leq 1  \right\}$.  If $\mathfrak{r} = \prod_{i=1}^k \mathfrak{r}_i$ satisfies $\mu(\mathfrak{r})^2 = 1$ and $(\mathfrak{r},\mathfrak{w})=1$, set
\[
y_{\mathfrak{r}_1, \dots, \mathfrak{r}_k} := F\left(\frac{\log |\mathfrak{r}_1|}{\log R}, \dots, \frac{\log |\mathfrak{r}_k|}{\log R}  \right)
\]
and set $y_{\mathfrak{r}_1, \dots, \mathfrak{r}_k} = 0$ otherwise.  In order to evaluate the summations of $y_{\mathfrak{r}_1,\dots,\mathfrak{r}_k}$, we will need the following lemma, which is an analogue of a result of Goldston, Graham, Pintz, and Y{\i}ld{\i}r{\i}m \cite[Lemma 4]{GGPY09}. (This result also appears as Lemma 6.1 in \cite{Maynard2013}.)

\begin{lemma} \label{lem:summation}
Suppose $\gamma$ is a multiplicative function on the nonzero ideals of $A$ such that there are constants $\kappa >0, A_1 > 0$, $A_2 \ge 1$, and $L \ge 1$ satisfying
\[
0 \leq \frac{\gamma(\mathfrak{p})}{\mathfrak{p}} \leq 1 - A_1,
\]
and
\[
-L \leq
\sum_{w \leq |\mathfrak{p}|< z}\frac{\gamma(\mathfrak{p}) \log |\mathfrak{p}|}{|\mathfrak{p}|} - \kappa \log z / w
\leq A_2,
\]
for any $2 \leq w \leq z$. Let $g$ be the totally multiplicative function defined on prime ideals by $g(\mathfrak{p}) = \gamma(\mathfrak{p}) / (|\mathfrak{p}| - \gamma(\mathfrak{p}))$. Let $G\colon [0,1] \to \mathbb{R}$ be a piecewise differentiable function, and let $G_{max} = \sup_{t \in [0,1]}(|G(t)| + |G'(t)|)$. Then
\[
\sum_{|\mathfrak{d}| < z} \mu(\mathfrak{d})^2 g(\mathfrak{d}) G \left(\frac{\log |\mathfrak{d}|}{\log z} \right)
=
\mathfrak{S}\frac{c_A^\kappa \cdot (\log z)^\kappa}{\Gamma(\kappa)} \int_0^1 G(x)x^{\kappa - 1} dx
+ O_{A,A_1, A_2, \kappa} \left(L G_{max} (\log z)^{\kappa - 1}
\right),
\]
where $c_A := \mathop{\mathrm{Res}}_{s=1} \zeta_A(s)$ and
\[
\mathfrak{S} =
\prod_{\mathfrak{p}} \left(1 - \frac{\gamma(\mathfrak{p})}{|\mathfrak{p}|} \right)^{-1}
\left(1 - \frac{1}{|\mathfrak{p}|} \right)^\kappa.
\]
\end{lemma}

\begin{remark} In both \cite{GGPY09} and \cite{Maynard2013}, the analogous error term is asserted to be $O(\mathfrak{S} L G_{\max} (\log{z})^{\kappa-1})$. In other words, there is a factor of $\mathfrak{S}$ not present in our statement. However, the proofs appear to support this stronger estimate only if one makes a further assumption on the size of $z$ compared to $L$. Fortunately, this discrepancy is of no importance in the applications, as this error term is always subsumed by larger errors.
\end{remark}

\begin{proof} Let $G(z):= \sum_{|\mathfrak{d}| < z} \mu(\mathfrak{d})^2 g(\mathfrak{d})$. A straightforward argument using partial summation (along the lines of that given explicitly by Goldston et al. in their proof of \cite[Lemma 4]{GGPY09}) reduces the claim to showing that
\[ G(z) = \mathfrak{S} \cdot \frac{c_{A}^{\kappa} (\log{z})^{\kappa}}{\Gamma(\kappa+1)} + O_{A, A_1, A_2, \kappa}(L (\log{(2z)})^{\kappa-1}) \]
for all $z \ge 1$. This last assertion is an exact analogue of what is shown by Halberstam and Richert in their proof of Lemma 5.4 in \cite{HalberstamRichert2011}. In fact, following their argument \cite[pp. 147--151]{HalberstamRichert2011} essentially verbatim, we find that
\[ G(z)= c (\log{z})^{\kappa} + O(L (\log{(2z)})^{\kappa-1}) \]
for some constant $c$ and all $z \ge 1$. (Compare with equations (3.10) and (3.11) on pages 150 and 151 of \cite{HalberstamRichert2011}.) It remains only to show that $c= c_A^{\kappa} \cdot \mathfrak{S}/{\Gamma(k+1)}$. The argument at the bottom of p. 151 of \cite{HalberstamRichert2011} shows that
\[ c = \frac{1}{\Gamma(\kappa+1)} \lim_{s\to 0^+} s^{\kappa} \prod_{p}\left(1 + \frac{g(\mathfrak{p})}{|\mathfrak{p}|^s}\right). \]
To compute the limit, note that $\zeta_A(s+1) = \prod_{\mathfrak{p}} (1-|\mathfrak{p}|^{-s-1})$ and that $s \sim c_A/\zeta_A(s+1)$ as $s\to 0^+$. This implies that
\[ \lim_{s\to 0^+} s^{\kappa} \prod_{\mathfrak{p}}\left(1 + \frac{g(\mathfrak{p})}{|\mathfrak{p}|^s}\right) = c_A^{\kappa} \cdot \lim_{s\to 0^+} \prod_{\mathfrak{p}}\left(1 + \frac{g(\mathfrak{p})}{|\mathfrak{p}|^s}\right) \left(1-\frac{1}{|\mathfrak{p}|^{s+1}}\right)^{\kappa}. \]
Our opening assumptions on $\gamma(\mathfrak{p})$ imply uniform convergence of the final product for real $s \geq 0$. (The proof of this follows the proof of the first part of Lemma 5.3 in \cite{HalberstamRichert2011}.) Thus, \[ \lim_{s\to 0^+} \prod_{\mathfrak{p}}\left(1 + \frac{g(\mathfrak{p})}{|\mathfrak{p}|^s}\right) \left(1-\frac{1}{|\mathfrak{p}|^{s+1}}\right)^{\kappa} = \prod_{p}\left(1 + {g(\mathfrak{p})}\right) \left(1-\frac{1}{|\mathfrak{p}|}\right)^{\kappa}= \mathfrak{S}. \] Hence, $c= c_A^{\kappa} \cdot \mathfrak{S}/{\Gamma(k+1)}$, which completes the proof of the lemma.
\end{proof}

\begin{proof}[Proof of Proposition \ref{prop:main}] With all of our earlier results in place, the proof of this result follows from exactly the same reasoning as Maynard's proofs of Lemmas 6.1 and 6.2 \cite{Maynard2013}. Here our Lemma \ref{lem:maynard_first_sum} replaces his Lemma 5.1, our Lemma \ref{lem:maynard_second_sum} replaces his Lemma 5.2, our Lemma \ref{lem:y^{(m)}} replaces his Lemma 5.3, and our Lemma \ref{lem:summation} replaces his Lemma 6.1. In fact, we find that the asymptotic estimates for $S_1$ and $S_2$ asserted in Proposition \ref{prop:main} hold with errors that are $O(F_{\max}^2 |A(N)| \frac{\phi(\mathfrak{w})^{k} (\log{R})^k}{|\mathfrak{w}|^{k+1} D_0})$.
\end{proof}

\subsection{Final assembly of theorems}
\label{subsec:winning}

Proposition \ref{prop:main} allows us to obtain the following analogue of \cite[Proposition 4.2]{Maynard2013}.
\begin{corollary}\label{cor:maincor} Suppose that the set of primes $P$ in $A$ has level of distribution $\theta > 0$. Let $\mathcal{H} = (h_1, \dots, h_k)$ be an admissible $k$-tuple. Let $I_k(F)$ and $J_k^{(m)}(F)$ be defined as in the statement of Proposition \ref{prop:main}. Let $\mathcal{S}_k$ denote the set of piecewise differentiable functions $F\colon [0,1]\to \mathbb{R}$ supported on $\mathcal{R}_k:= \{(x_1, \dots, x_k)\in [0,1]^k: \sum_{i=1}^{k} x_i \leq 1\}$ with $I_k(F) \neq 0$ and $J_k^{(m)}(F) \neq 0$ for each $m$. Let
	\[ M_k := \sup_{F \in \mathcal{S}_k} \frac{\sum_{m=1}^{k} J_k^{(m)}(F)}{I_k(F)} \quad\text{and let}\quad r_k := \left\lceil \frac{\theta M_k}{2}\right\rceil.\]
There are infinitely many $\alpha \in A$ such that at least $r_k$ of the $\alpha+h_i$ ($1 \leq i \leq k$) are prime.
\end{corollary}

\begin{proof} We mimic the proof of \cite[Proposition 4.2]{Maynard2013}. Recall from \S\ref{subsec:sieve} that if $S:= S_2 - \rho S_1 > 0$ for a certain $N$, then there are more than $\rho$ primes among the $\alpha+h_i$ ($1 \leq i \leq k$), for some $\alpha \in A(N)$. Consequently, if $S > 0$ for all large $N$, then there are infinitely many translates of $(h_1, \dots, h_k)$ containing more than $\rho$ primes. Put $R = |A(N)|^{\theta/2- \epsilon}$ for a small $\epsilon > 0$. Choose $F_0 \in \mathcal{S}_K$ so that $\sum_{m=1}^{k} J_k^{(m)}(F_0) > (M_k-\epsilon) I_k(F_0)$. Using Proposition \ref{prop:main}, we see we can choose the weights $\lambda_{\mathfrak{d}_1, \dots, \mathfrak{d}_k}$ so that
	\begin{align*} S &= \frac{\phi(\mathfrak{w})^k}{|\mathfrak{w}|^{k+1}} |A(N)| (c_A \log{R})^k \left( \frac{(c_A \log{R}) |P(N)|}{|A(N)|} \sum_{m=1}^{k} J_k^{(m)}(F) - \rho I_k(F_0) + o(1)\right) \\
	&\ge  \frac{\phi(\mathfrak{w})^k}{|\mathfrak{w}|^{k+1}} |A(N)| (c_A \log{R})^k I_k(F_0) \left( \Delta \cdot  \left(\frac{\theta}{2}-\epsilon\right)\left(M_k-\epsilon\right) - \rho + o(1)\right), \end{align*}
where
\[ \Delta:= c_A \cdot \lim_{N\to\infty} \frac{|P(N)| \log |A(N)|}{|A(N)|}. \]
(The existence of this limit will be shown momentarily.) If $\rho = \Delta \cdot \Theta \cdot M_k/2 - \delta$, then choosing $\epsilon$ sufficiently small, we get that $S > 0$ for large $N$. Since $\delta > 0$ was arbitrary, there must be infinitely many $\alpha \in A$ such that at least $\lceil \Delta \cdot \Theta \cdot M_k/2\rceil$ of the $\alpha+h_i$ ($1 \leq i \leq k$) are prime.

We now show that $\Delta=1$, which will complete the proof of the proposition. We consider separately the cases when $A=\mathbb{F}_q[t]$ and when $A =\mathcal{O}_K$.

If $A= \mathbb{F}_q[t]$, then $\zeta_A(s) = \frac{1}{1-q^{1-s}}$ and so $c_A = \frac{1}{\log{q}}$. On the other hand, for $N=q^n$, we have $|A(N)| = q^n$ and $|P(N)| = q^{n}/n + O(q^{n/2}/n)$. (For all of these facts, see, e.g., \cite[pp. 11--14]{Rosen2002}.) Thus, $|P(N)| \sim |A(N)| \log{q}/\log|A(N)|$ as $N=q^n \to\infty$, and so $\Delta=1$.

Suppose now that $A = \mathcal{O}_K$, where $K$ is a number field with $r_1$ real embeddings and $r_2$ pairs of complex conjugate embeddings. Consider the region in Minkowski space corresponding to the conditions defining $A_0(N)$:
\begin{multline*} \{(x_1, \dots, x_{r_1}, z_{r_1+1}, \dots, z_{r_1+r_2}) \in \mathbb{R}^{r_1} \times \mathbb{C}^{r_2}: 0 \leq x_i \leq N, |z_j| \leq N \\ \text{ for all $1\le i\le r_1$ \text{ and } $r_1+1 \le j\le r_1+r_2$}\}.  \end{multline*}
This has volume $N^d \cdot \pi^{r_2}$. On the other hand, the image of $\mathcal{O}_K$ under the Minkowski embedding is a lattice with covolume $2^{-r_2} \sqrt{|D_K|}$, where $D_K$ denotes the discriminant of $K$. It follows that $|A_0(N)| \sim {(2\pi)^{r_2} N^d}/{\sqrt{|D|}}$, as $N\to\infty$. Since $A(N) = A_0(2N) \setminus A_0(N)$,
\[ |A(N)| \sim \frac{(2\pi)^{r_2} (2N)^d (1-1/2^d)}{\sqrt{|D|}}. \]
We turn now to the estimation of $|P(N)|$. For this, we employ Mitsui's generalized prime number theorem \cite{Mitsui1956}, a special case which is that the number of primes in $A_0(N)$ is
\[ \sim \frac{w_K}{2^{r_1} h_K \mathrm{Reg}_K} \idotsint\limits_{[2,N]^{r_1}\times [2,N^2]^{r_2}} \frac{du_1 \dots du_{r_1+r_2}}{\log(u_1 \dots u_{r_1+r_2})}, \]
as $N \to \infty$. Here $w_K$ is the number of roots of unity contained in $K$, $h_K$ is the class number of $K$, and $\mathrm{Reg}_K$ is the regulator of $K$. The integral appearing here is asymptotic to $N^d/\log{(N^d)}$, by  \cite[Lemma 6]{Hinz1981}. Hence,
\[ |P(N)| \sim \frac{w_K}{2^{r_1} h_K \mathrm{Reg}_K} (1-1/2^d) \frac{(2N)^d}{\log((2N)^d)}.  \]
Finally, Dedekind's class number formula asserts that
\[ c_A = \frac{2^{r_1} (2\pi)^{r_2} h_K \mathrm{Reg}_K}{w_K \sqrt{|D_K|}}. \]
Referring back to the definition of $\Delta$, we find after some algebra that indeed $\Delta=1$.
\end{proof} 

As shown by Maynard \cite[Proposition 4.13]{Maynard2013}, we have $M_k > \log{k}-2\log\log{k}-2$ for all large enough values of $k$. In particular, $M_k\to\infty$ as $k\to\infty$. So Theorem \ref{thm:number_field} follows at once from Corollary \ref{cor:maincor} provided that the primes in $\mathcal{O}_K$ always possess a positive level of distribution. Similarly, Theorem \ref{thm:function_field} follows provided that the primes in $\mathbb{F}_q[t]$ possess a positive level of distribution not depending on $q$. Both provisos were already asserted to hold in \S\ref{subsec:dictionary}. In fact, we have the following:

\begin{theorem}[Hinz]\label{thm:hinz} Let $K/\mathbb{Q}$ be a number field with $r_2$ pairs of complex conjugate embeddings. If $r_2=0$ (i.e., $K$ is totally real), the set $P$ of primes of $\mathcal{O}_K$ has level of distribution $\theta$ for any $\theta < \frac{1}{2}$. In general, $P$ has level of distribution $\theta$ for any $\theta < \frac{1}{r_2 + \frac{5}{2}}$.
\end{theorem}

\begin{theorem}\label{thm:hayes} If $A= \mathbb{F}_q[t]$, then the set $P$ has level of distribution $\frac{1}{2}$. Indeed, for all $\mathfrak{q}$ and $N$, we have the (stronger) pointwise error estimate
	\[ \max_{\substack{\alpha_0\bmod{\mathfrak{q}}\\ (\alpha_0, \mathfrak{q})=1 }}|\mathcal{E}(N;\mathfrak{q},\alpha_0)| \ll (\log{2|\mathfrak{q}|})\cdot |A(N)|^{1/2}. \]
\end{theorem}

Theorem \ref{thm:hinz} is contained in the somewhat more general main theorem of \cite{Hinz1988}. Theorem \ref{thm:hayes} is a consequence of Weil's Riemann Hypothesis and was first deduced by Hayes \cite{Hayes1965} (see also \cite{Car1999}).

\section{Applications}
\label{sec:applications}

\subsection{Bounded gaps in totally real number fields} The proof of Corollary \ref{cor:bounded_real} makes use of the following simple observation.
\label{subsec:real}

\begin{lemma}\label{lem:admissibleinZ} Suppose that $\mathcal{H}$ is an admissible tuple in $\mathbb{Z}$. Then $\mathcal{H}$ is also an admissible tuple in $\mathcal{O}_K$ for every number field $K$.
\end{lemma}
\begin{proof} The reduction of $\mathcal{H}$ modulo $\mathfrak{p}$ always lands in the prime subfield of $\mathcal{O}_K/\mathfrak{p}$ and so cannot cover $\mathcal{O}_K/\mathfrak{p}$ unless $\mathfrak{p}$ has degree $1$. But if $\mathfrak{p}$ has degree $1$, then $\mathcal{O}_K/\mathfrak{p} \cong \mathbb{Z}/p\mathbb{Z}$ for some rational prime $p$, and $\mathcal{H}$ fails to cover $\mathbb{Z}/p\mathbb{Z}$ since $\mathcal{H}$ is admissible in $\mathbb{Z}$.
\end{proof}

Suppose now that $K$ is totally real. By Theorem \ref{thm:hinz}, the primes in $K$ have level of distribution $\theta$ for any $\theta < \frac{1}{2}$. Maynard \cite[Proposition 4.3]{Maynard2013} has shown that the number $M_{105}$ in Corollary \ref{cor:maincor} satisfies $M_{105} > 4$. Corollary \ref{cor:bounded_real} follows now from Corollary \ref{cor:maincor}, Lemma \ref{lem:admissibleinZ}, and the result of Engelsma that there exists an admissible $105$-tuple $h_1 < h_2 < \dots < h_{105}$ of rational integers with $h_{105}-h_1=600$.

\subsection{Gap densities in $\mathbb{F}_q(t)$}
\label{subsec:fun}

We now turn our attention to Theorem \ref{thm:functionfield_proportion}, which we recall concerns gaps between monic irreducibles of fixed large degree $n$.

\begin{proof}[Proof of Theorem \ref{thm:functionfield_proportion}]  Let $k_0:=k_0(2)$ from Theorem \ref{thm:function_field}, and assume that $q\geq k_0+1$.

\emph{(i)} We wish to show that any monomial $a\cdot t^d \in \mathbb{F}_q[t]$ occurs as a gap between monic irreducibles of degree $n$ for every sufficiently large $n$ satisfying $(n-d,q-1)=1$.

For any $q$, the tuple $\{\alpha t^d : \alpha\in\mathbb{F}_q^\times\}$ is admissible, and so, because $q\geq k_0+1$, we may apply Theorem \ref{thm:function_field}.  We thus see that, for each sufficiently large $n$, some monomial $c\cdot t^d$ occurs as a gap between monic irreducibles of degree $n$; call these irreducibles $f_1(t)$ and $f_2(t)$.  If $(n-d,q-1)=1$, there is an $\omega\in\mathbb{F}_q^\times$ such that $\omega^{n-d}=c/a$, and we note that the polynomials $f_1(\omega t)/\omega^n$ and $f_2(\omega t)/\omega^n$ are monic and irreducible.  We then compute that
\[
\frac{f_1(\omega t)}{\omega^n} - \frac{f_2(\omega t)}{\omega^n} = \frac{c \cdot (\omega t)^d}{\omega^n} = \frac{c}{\omega^{n-d}} \cdot t^d = a\cdot t^d.
\]

\emph{(ii)} We now turn our attention to the second part of Theorem \ref{thm:function_field} concerning the proportion of degree $d$ polynomials that appear as gaps in degree $n$.

Let $Z(k,d,n)$ denote the assertion that, for any admissible $k$-tuple $(h_1,\dots,h_k)$ such that each of $h_1,\dots,h_k$ and $h_1-h_2,h_1-h_3,\dots,h_{k-1}-h_k$ is of degree $d$, there is an $f\in\mathbb{F}_q[t]$ of degree $n$ such that at least two of $f+h_1,\dots,f+h_k$ are monic and irreducible; we note that Theorem \ref{thm:function_field} implies that $Z(k_0,d,n)$ holds for any $d$ provided that $n$ is sufficiently large.  We will prove by induction on $k \leq k_0$ that if $Z(k,d,n)$ holds, then the proportion of polynomials of degree $d$ appearing as gaps in degree $n$ is at least
$
\frac{1}{k-1}-\frac{1}{q-1}.
$

If $k=2$, the assertion is clear: $Z(2,d,n)$ implies that every non-zero polynomial of degree $d$ appears as a gap. For $k \geq 3$, we note that either $Z(k-1,d,n)$ holds or it doesn't.  If we are in the former case, then, as $1/(k-1)$ is decreasing, the conclusion follows.  On the other hand, if $Z(k-1,d,n)$ does not hold, then there must be $h_1,\dots,h_{k-1}$ as above for which there is no $f\in\mathbb{F}_q[t]$ of degree $n$ such that two of $f+h_1,\dots,f+h_{k-1}$ are monic irreducibles.  Now $q\geq k_0+1 > k$; thus, for {\bf any} $h$ of degree $d$ with each difference $h-h_1,\dots,h-h_{k-1}$ also of degree $d$, the tuple $(h_1,\dots,h_{k-1},h)$ is admissible. Since we are assuming $Z(k, d, n)$ holds, there must be an $f$ of degree $n$ for which $f+h$ and some $f+h_i$ are both monic irreducibles; hence, $h-h_i$ occurs as a gap.  Varying over the $(q-1-(k-1))\cdot q^d$ such $h$, each gap can appear at most $k-1$ times, whence the number of distinct gaps is at least
\[
\frac{(q-1-(k-1))\cdot q^d}{k-1}.
\]
Noting that there are $q^d\cdot(q-1)$ elements of degree $d$, the claim follows.  Lastly, the assertion about monomials comes from only looking at tuples $(h_1,\dots,h_k)$ with each $h_i$ a distinct monomial of degree $d$.
\end{proof}

\section*{Acknowledgements}
This work began at the American Institute of Mathematics workshop on arithmetic statistics over finite fields and function fields. We would like to thank AIM for providing the opportunity for us to work together.

\bibliographystyle{abbrv}
\bibliography{bibliography}

\begin{thebibliography}{10}

\bibitem{Car1999}
M.~Car.
\newblock Distribution des polyn\^omes irr\'eductibles dans {${\bf F}_q[T]$}.
\newblock {\em Acta Arith.}, 88(2):141--153, 1999.

\bibitem{FI2010}
J.~Friedlander and H.~Iwaniec.
\newblock {\em Opera de cribro}, volume~57 of {\em American Mathematical
  Society Colloquium Publications}.
\newblock American Mathematical Society, Providence, RI, 2010.

\bibitem{GGPY09}
D.~A. Goldston, S.~W. Graham, J.~Pintz, and C.~Y. Y{\i}ld{\i}r{\i}m.
\newblock Small gaps between products of two primes.
\newblock {\em Proc. Lond. Math. Soc. (3)}, 98(3):741--774, 2009.

\bibitem{GPY2009}
D.~A. Goldston, J.~Pintz, and C.~Y. Y{\i}ld{\i}r{\i}m.
\newblock Primes in tuples. {I}.
\newblock {\em Ann. of Math. (2)}, 170(2):819--862, 2009.

\bibitem{HalberstamRichert2011}
H.~Halberstam and H.-E. Richert.
\newblock {\em Sieve Methods}.
\newblock Dover books on mathematics. Dover Publications, 2011.

\bibitem{Hall2006}
C.~Hall.
\newblock {$L$}-functions of twisted {L}egendre curves.
\newblock {\em J. Number Theory}, 119(1):128--147, 2006.

\bibitem{Hayes1965}
D.~R. Hayes.
\newblock The distribution of irreducibles in {${\rm GF}[q,\,x]$}.
\newblock {\em Trans. Amer. Math. Soc.}, 117:101--127, 1965.

\bibitem{Hinz1981}
J.~G. Hinz.
\newblock On the theorem of {B}arban and {D}avenport-{H}alberstam in algebraic
  number fields.
\newblock {\em J. Number Theory}, 13(4):463--484, 1981.

\bibitem{Hinz1988}
J.~G. Hinz.
\newblock A generalization of {B}ombieri's prime number theorem to algebraic
  number fields.
\newblock {\em Acta Arith.}, 51(2):173--193, 1988.

\bibitem{Kaptan2013}
D.~A. Kaptan.
\newblock A generalization of the {G}oldston-{P}intz-{Y}ildirim prime gaps
  result to number fields.
\newblock {\em Acta Math. Hungar.}, 141(1-2):84--112, 2013.

\bibitem{Maynard2013}
J.~Maynard.
\newblock Small gaps between primes.
\newblock {\em Preprint available at http://arxiv.org/abs/1311.4600}, 2013.

\bibitem{Mitsui1956}
T.~Mitsui.
\newblock Generalized prime number theorem.
\newblock {\em Jap. J. Math.}, 26:1--42, 1956.

\bibitem{Pollack2008}
P.~Pollack.
\newblock An explicit approach to {H}ypothesis {H} for polynomials over a
  finite field.
\newblock In {\em Anatomy of integers}, volume~46 of {\em CRM Proc. Lecture
  Notes}, pages 259--273. Amer. Math. Soc., Providence, RI, 2008.

\bibitem{Polymath}
Polymath8.
\newblock Bounded gaps between primes.
\newblock
  \url{http://michaelnielsen.org/polymath1/index.php?title=Bounded_gaps_betwee%
n_primes}.

\bibitem{Rosen99}
M.~Rosen.
\newblock A generalization of {M}ertens' theorem.
\newblock {\em J. Ramanujan Math. Soc.}, 14(1):1--19, 1999.

\bibitem{Rosen2002}
M.~Rosen.
\newblock {\em Number theory in function fields}, volume 210 of {\em Graduate
  Texts in Mathematics}.
\newblock Springer-Verlag, New York, 2002.

\bibitem{Thorner2014}
J.~Thorner.
\newblock Bounded gaps between primes in {C}hebotarev sets.
\newblock {\em Preprint available at http://arxiv.org/abs/1401.6677}, 2014.

\bibitem{Zhang2013}
Y.~Zhang.
\newblock Bounded gaps between primes.
\newblock {\em Ann. of Math.}, 179:1121--1174, 2014.

\end{thebibliography}

\end{document}